\documentclass[english]{extarticle}

\usepackage[T1]{fontenc}
\usepackage[latin9]{inputenc}
\usepackage[a4paper]{geometry}
\usepackage{color}
\usepackage{babel}
\usepackage{float}
\usepackage{amsmath}
\usepackage[all]{xy}
\usepackage{amsthm}
\usepackage{amssymb}
\usepackage[numbers]{natbib}
\usepackage[unicode=true,pdfusetitle,
 bookmarks=true,bookmarksnumbered=false,bookmarksopen=false,
 breaklinks=false,pdfborder={0 0 1},backref=page,colorlinks=true]
 {hyperref}
\hypersetup{
 linkcolor=blue, urlcolor=marineblue, citecolor=blue, pdfstartview={FitH}, hyperfootnotes=false, unicode=true}

\makeatletter
\numberwithin{equation}{section}
\numberwithin{figure}{section}
\theoremstyle{plain}
\newtheorem{thm}{\protect\theoremname}
\theoremstyle{definition}

\theoremstyle{definition}

\theoremstyle{plain}
\newtheorem{lemma}{Lemma}

\def \De{\Delta}

\def \Om{\Omega}

\def \th{\theta}

\def \AAA{{\cal A}}
\def \RR{{\mathbb R}}
\def \DD{{\mathbb D}}
\def \CC{{\mathbb C}}
\def \HH{{\cal H}}
\def \ff{\infty}
\usepackage{lmodern}
\usepackage[T1]{fontenc}
\usepackage{concmath}
\usepackage{ae,aecompl}
\usepackage[T1]{fontenc}
\usepackage{amsmath}
\usepackage{graphicx}

\usepackage{tikz}

\numberwithin{equation}{section}
\numberwithin{figure}{section}

\makeatother

\providecommand{\definitionname}{Definition}
\providecommand{\examplename}{Example}
\providecommand{\theoremname}{Theorem}
\def \vski{\vspace{12pt}}

\def \th{\theta}

\begin{document}
\title{Remarks on Gross' technique for obtaining a conformal Skorohod embedding of planar Brownian motion}

\author{Maher Boudabra, Greg Markowsky\\
Monash University}

\maketitle

\begin{abstract}
In \citep{gross2019conformal} it was proved that, given a distribution
$\mu$ with zero mean and finite second moment, there exists a simply
connected domain $\Omega$ such that if $Z_{t}$ is a standard planar
Brownian motion, then $\mathcal{R}e(Z_{\tau_{\Omega}})$ has the distribution $\mu$.
In this note, we extend this method to prove that if $\mu$
has a finite $p$-th moment then the exit time $\tau_{\Omega}$ has
a finite moment of order $\frac{p}{2}$. We also prove a uniqueness principle for this construction, and use it to give several examples.
\end{abstract}

\section{Introduction and statement of results}

In what follows, $Z_{t}$ is a standard planar Brownian motion starting at 0, and for any plane domain $\Omega$ containing 0 we let $\tau_{\Omega}$ denote the first exit time of $Z_t$ from $\Omega$. In the elegant recent paper \citep{gross2019conformal} the following theorem was proved.

\begin{thm}
Given a probability distribution $\mu$ on $\RR$ with zero mean and finite nonzero second moment,
we can find a simply connected domain $\Omega$ such that $\mathcal{R}e(Z_{\tau_{\Omega}})$ has
the distribution $\mu$. Furthermore we have $E[\tau_{\Omega}]<\ff$.
\end{thm}

We will prove several new results related to Gross' construction. Our first result is the following generalization.

\begin{thm} \label{newguy}
Given a probability distribution $\mu$ on $\RR$ with zero mean and finite nonzero $p$-th moment (with $1<p<\ff$),
we can find a simply connected domain $\Omega$ such that $\mathcal{R}e(Z_{\tau_{\Omega}})$ has
the distribution $\mu$. Furthermore we have $E[(\tau_{\Omega})^{p/2}]<\ff$.
\end{thm}

The proof of this result depends on a number of known properties of the Hilbert transform and of the exit time $\tau_{\Omega}$, and is rather short. However the results needed are scattered through a number of different subfields of probability and analysis, and in an attempt to make the paper self-contained we have included a certain amount of exposition on these topics. We will prove the theorem in the next section.

\vski

There are several reasons why we feel that our extension is worth noting. The moments of $\tau_{\Omega}$ have special importance in two dimensions, as they carry a great deal of analytic and geometric information about the domain $\Omega$. The first major work in this direction seems to have been by Burkholder in \cite{burkholder1977exit}, where it was proved among other things that finiteness of the $p$-th Hardy norm of $\Omega$ is equivalent to finiteness of the $\frac{p}{2}$-th moment of $\tau_{\Omega}$. To be precise, for any simply connected domain $\Omega$
let

$$
{\rm H}(\Omega)=\sup \{p > 0: {\bf E}[(\tau_{\Omega})^p] < \ff\};
$$

note that ${\rm H}(\Omega)$ is proved in \cite{burkholder1977exit} to be exactly
equal to half of the Hardy number of $\Omega$, as defined in \cite{hansen}, which is defined to be

$$
{\tilde{\rm H}}(\Omega)=\sup \{q > 0: \lim_{r \nearrow 1} \int_{0}^{2\pi} |f(re^{i\th})|^qd \th < \ff\},
$$

where $f$ is a conformal map from the unit disk onto $\Omega$. This equivalence was used in \cite{burkholder1977exit} to show for instance that $H(W_\alpha) = \frac{\pi}{2 \alpha}$, where $W_\alpha = \{0<Arg(z)<\alpha\}$ is an infinite angular wedge with angle $\alpha$. In fact, coupled with the purely analytic results in \cite{hansen} this can be used to determine ${\rm H}(\Omega)$ for any starlike domain $\Omega$. If we assume that $V$ is starlike with respect to $0$, then we may define

\begin{equation} \label{bigmax}
\AAA_{r,\Omega} = \max \{m(E): E \mbox{ is a subarc of } \Omega \cap \{|z|=r\}\},
\end{equation}

and this quantity is non-increasing in $r$ (here $m$ denotes angular Lebesgue measure on the circle). We may therefore let $\AAA_\Omega = \lim_{r \nearrow \ff} \AAA_{r,\Omega}$, and then combining the results in \cite{hansen} and \cite{burkholder1977exit} (see also \cite{markowsky}) we have ${\rm H}(\Omega) = \frac{\pi}{2 \AAA_\Omega}$. In this sense, the quantity ${\rm H}(\Omega)$ provides us with some sort of measure of the aperture of the domain at $\ff$. Also in \cite{markowsky}, a version of the Phragm\'en-Lindel\"of principle was proved that makes use of the quantity ${\rm H}(\Omega)$. Furthermore, the quantity ${\bf E}[(\tau_{\Omega})^p]$ provides us with an estimate for the tail probability $P(\tau_{\Omega} > \delta)$: by Markov's inequality, $P(\tau_{\Omega} > \delta) \leq \frac{{\bf E}[(\tau_{\Omega})^p]}{\delta^p}$.

\vski

For these reasons, we would argue that Theorem \ref{newguy} gives a partial answer to the following intriguing question posed by Gross in \citep{gross2019conformal}: given a probability distribution $\mu$ and a corresponding $\Omega$ such that $\mathcal{R}e(Z_{\tau_{\Omega}})$ has distribution $\mu$, in what sense are properties of $\mu$ reflected in the geometric properties of $\Omega$? We will have more comments on this question in the final section.

\vski

Our next result is influenced by Gross' observation that the domain corresponding to a given measure $\mu$ is not unique. Without further conditions this is correct, however we have found that natural conditions can be imposed on the domain so that a uniqueness principle holds. Before stating the result, let us make some definitions. A domain $U$ is {\it symmetric} if $\bar z \in U$ whenever $z \in U$. We will call a symmetric domain $U$ {\it $\De$-convex} if the vertical line segment connecting $z$ and $\bar z$ lies entirely within $U$ for any $z \in U$. It is straightforward to verify that any domain constructed by Gross' technique is both symmetric and $\De$-convex (see Section \ref{hilbert}), and we have the following result.

\begin{thm} \label{thm:The-distribution-}
For any distribution $\mu$ satisfying the conditions of Theorem \ref{newguy}, there is a unique simply connected domain $\Omega$ such that $Re(Z_{\tau_\Om}) \sim \mu$ and which is symmetric, $\De$-convex, and satisfies $E[(\tau_\Om)^{p/2}] < \ff$.
\end{thm}

The importance of this result for our purposes is that it allows us to give a number of examples of domains generated by Gross' method. That is, if $\Om$ is a simply connected domain which is symmetric, $\De$-convex, and satisfies $E[(\tau_\Om)^{p/2}] < \ff$, then it must be the domain generated by Gross' method corresponding to the distribution of $Re(Z_{\tau_\Om})$. We will exploit this fact in Section \ref{examples}.

\section{Preliminaries and proof of Theorem \ref{newguy}} \label{hilbert}

The proof of Theorem \ref{newguy} is mainly based on the Hilbert transform
theory for periodic functions, and we give here a brief summary of this. For further details about the topic, we refer the reader to \citep{butzer1971hilbert}.

The Hilbert transform of a $2\pi$- periodic function $f$ is defined
by
\[
\HH_{f}(x):=PV\left\{ \frac{1}{2\pi}\int_{-\pi}^{\pi}f(x-t)\cot(\frac{t}{2})dt\right\} =\lim_{\eta\rightarrow0}\frac{1}{2\pi}\int_{\eta\leq|t|\leq\pi}f(x-t)\cot(\frac{t}{2})dt
\]
where $PV$ denotes the Cauchy principal value, which is required
here as the trigonometric function $t\longmapsto\cot(\cdot)$ has
a simple pole at $k\pi$ with $k\in\mathbb{Z}$. Note that the more standard Hilbert
transform is defined for functions $f$ on the real line by

\[
PV\left\{ \frac{1}{2\pi}\int_{-\infty}^{+\infty}\frac{f(x-t)}{t}dt\right\} .
\]

However, replacing $\frac{1}{t}$ by $\cot(\frac{t}{2})$ in the integrand is natural because $\cot(\frac{t}{2})$ is the function which results by "wrapping" $\frac{1}{t}$ around the circle; to be precise, $\cot(\cdot)$
satisfies the following identity (\citep{remmert2012theory}):

\[
\cot(z)=\frac{1}{z}+2z\sum_{n=1}^{+\infty}\frac{1}{z^{2}-n^{2}}=\frac{\pi}{z}+\pi\sum_{n=1}^{+\infty}\left(\frac{1}{z+\pi n}+\frac{1}{z-\pi n}\right)
\]

In this sense, $\cot(\frac{t}{2})$ can be seen as the periodic version of the function $\frac{1}{t}$. Let us now sketch the ideas for Gross' proof, so that we may see where the Hilbert transform comes in. We assume for now that $\mu$ has finite second moment. Let $F$ be the c.d.f of $\mu$
and consider the pseudo-inverse function of $F$ defined by
\[
G(u):=\inf\{x\in\mathbb{R}|F(x)\geq u\}.
\]

Note that $G$ is defined for $u \in[0,1]$. It is well known that $G(\mathrm{Uni}_{(0,1)})$ has $\mu$ as distribution.
Now consider the $2\pi$-periodic function $\varphi$ whose restriction
to $(-\pi,\pi)$ is $G(\frac{|\theta|}{\pi})$. The
map $\varphi$ is even, increasing on $(0,\pi)$ and belongs to $L^{2}$, where $L^p$ here and elsewhere in the paper denotes $L^p([-\pi,\pi])$.
Thus its Fourier series is well defined and converges to $\varphi$ in $L^{2}$.
We obtain hence
\[
\varphi(\theta)=\sum_{n=1}^{+\infty}\widehat{\varphi}(n)\cos(n\theta)\,\,
\]
where the $n$-th Fourier coefficient $\widehat{\varphi}(n)$ is defined
for all non negative integers $n$ by $\widehat{\varphi}(n)=\frac{1}{2\pi}\int_{0}^{2\pi}f(t)\cos(nt)dt$. It is clear that this is the real part of the power series generated by the Fourier coefficients $\widetilde{\varphi}(z):=\sum_{n=1}^{+\infty}\widehat{\varphi}(n)z^{n}$ evaluated at $z=e^{i\theta}$; that
is

\begin{equation}
\mathcal{R}e(\widetilde{\varphi}(e^{i \theta}))=\varphi(\theta)\label{real part}
\end{equation}

Note that $\mathcal{I}m(\widetilde{\varphi}(e^{i \theta}))$ is given by $\varphi(\theta)=\sum_{n=1}^{+\infty}\widehat{\varphi}(n)\sin(n\theta)$, and this is the Hilbert transform of $\varphi$. A crucial property of $\widetilde{\varphi}$, as is shown in \cite{gross2019conformal}, is that it is univalent. The image domain $\Omega:=\widetilde{\varphi}(\mathbb{D})$ is therefore simply
connected, and it is also symmetric over the $x$-axis as $\widetilde{\varphi}(\overline{z})=\overline{\widetilde{\varphi}(z)}$. Using the conformal invariance
of $Z_{t}$, and the fact that $Z_{\tau_{\mathbb{D}}}$ is uniformly
distributed on the boundary of $\partial\mathbb{D}$, we conclude by
(\ref{real part}) that $\mathcal{R}e(\widetilde{\varphi}(Z_{\tau_{\mathbb{D}}}))$ has distribution $\mu$. Furthermore, Parseval's identity and martingale theory implies that $E[\tau_\Omega] = \frac{1}{2} \sum_{n=1}^{\ff} |\widehat{\varphi}(n)|^2$ (see \cite{drum}), and this sum is finite since $\varphi \in L^{2}$.

\vski

Let us now see how we can extend this argument to prove Theorem \ref{newguy}. We will assume now that $\mu$ has a finite $p$-th moment, where $p>1$. It follows as above that $\varphi \in L^p$. The Fourier series $\sum_{n=1}^{+\infty}\widehat{\varphi}(n)\cos(n\theta)\,\,$ is still well defined and converges to $\varphi$ in $L^{p}$ (\cite[Thm. 3.5.7]{graf}). Parseval's identity is no longer available to us, but the following result allows us to conclude that the Hilbert transform $\sum_{n=1}^{+\infty}\widehat{\varphi}(n)\sin(n\theta)\,\,$ of $\varphi$ is also in $L^p$:

\begin{thm} \label{Hilb bound}
\citep{butzer1971hilbert} If $f$ is in $L^{p}$ then its periodic Hilbert transform $\HH_{f}$
does exist almost everywhere and we have
\begin{equation}
||\HH_{f}||_{L^{p}}\leq\lambda_{p}||f||_{L^{p}}\label{strong inequality}
\end{equation}
for some positive constant $\lambda_{p}$.
\end{thm}



We see that, as its real and imaginary parts are in $L^p$, the analytic function $\tilde \varphi(z)=\sum_{n=1}^{+\infty}\widehat{\varphi}(n)z^n\,\,$ lies in the Hardy space $H^p$, which is the space of all holomorphic maps on the disk with finite Hardy $p$-norm, defined as

\[
||f||_{H^{q}}:=\left\{ \lim_{r \nearrow 1} \frac{1}{2\pi}\int_{0}^{2\pi}|f(re^{it})|^{q}dt\right\} ^{\frac{1}{q}}.
\]

$\tilde \varphi(z)$ is also injective, by the same argument as was used in \cite{gross2019conformal}, and therefore $\Omega=\tilde \varphi(\DD)$ is simply connected. By a theorem of Burkholder in \citep{burkholder1977exit}
we have that if $f$ is a conformal function on the unit disk then
the following equivalence holds:
\begin{equation}
\tau_{f(\mathbb{D})}\in L^{\frac{p}{2}}\,\,\Longleftrightarrow\,\,||f||_{H^{p}}<\infty\label{Hardy}
\end{equation}

We see therefore that $E[(\tau_{\Omega})^{p/2}]<\ff$, and the theorem is proved.

\section{Proof of Theorem \ref{thm:The-distribution-}.}



\vski

In this section we prove Theorem \ref{thm:The-distribution-}, that the domain $\Omega$ generated by Gross' technique is the unique symmetric, $\De$-convex simply connected domain with $E[(\tau_{\Omega})^{p/2}]<\ff$ such that such that $\mathcal{R}e(Z_{\tau_{\Omega}})$ has
the distribution $\mu$. Before going through the proof, we need the following lemma related
to the Riemann mapping theorem.

\begin{lemma}
If $U \subsetneq \CC$ is a symmetric simply connected domain containing 0 then there exists
a conformal map from $\mathbb{D}$ to $U$ such that $f(0) = 0$ and $f((-1,1)) \subseteq \RR$.
\end{lemma}

\begin{proof}
The existence of a conformal map, say $f$, from the unit disc to
$U$ and sending zero to itself is guaranteed by the Riemann
mapping theorem. It remains to add the constraint that $f((-1,1)) \subseteq \RR$. Consideration of the power series shows that the map $\overline{f}(\overline{z})$ is analytic, and as $\DD$ and $U$ are symmetric it is a conformal map from $\DD$ to $U$. Therefore it is related to $f$ via a rotation acting on the unit
disc, that is

\[
\overline{f}(\overline{z})=f(e^{i\theta}z)
\]

for some $\theta\in[0,2\pi)$. The map $\widetilde{f}:z\longmapsto f(e^{i\frac{\theta}{2}}z)$
satisfies the requirement of the lemma since
\[
\begin{alignedat}{1}\overline{\widetilde{f}}(\overline{z}) & =\overline{f}(e^{i\frac{\theta}{2}}\overline{z})\\
 & =\overline{f}(\overline{e^{-i\frac{\theta}{2}}z})\\
 & =f(e^{i\theta}e^{-i\frac{\theta}{2}}z)\\
 & =f(e^{i\frac{\theta}{2}}z)\\
 & =\widetilde{f}(z)
\end{alignedat}
\]
In particular, if $z$ is real then $\widetilde{f}(z)$ is as well, which
ends the proof.
\end{proof}
We proceed now to prove Theorem \ref{thm:The-distribution-}. Suppose $U$ and $V$ are two domains satisfying the conditions of the theorem. Let $f:\mathbb{D}\longrightarrow U$
and $g:\mathbb{D}\longrightarrow V$ be two conformal maps fixing $0$ and sending reals
to reals. As $f$ and $g$ are injective, they are monotone on the real line, and we may assume then that they are increasing (if not, consider $f(-z)$ and/or $g(-z)$ instead). 
The power series $f(z)=\sum_{n=1}^{+\infty}a_{n}z^{n}$ and $g(z)=\sum_{n=1}^{+\infty}b_{n}z^{n}$
have real coefficients since $a_{n}=\frac{f^{(n)}(0)}{n!}\in\mathbb{R}$
and $b_{n}=\frac{g^{(n)}(0)}{n!}\in\mathbb{R}$. The fact that $E[(\tau_U)^{p/2}],E[(\tau_V)^{p/2}] < \ff$ implies that $||f||_{H^p},||g||_{H^p} < \ff$ (again, see \citep{burkholder1977exit}), and therefore the functions $f$ and $g$ have radial limits defined a.e. on $\{|z|=1\}$. That is, $f(e^{i \th}) := \lim_{r \nearrow 1} f(re^{i \th})$ exists for Lebesque almost every $\th$ on $[-\pi, \pi]$ (see \cite[Thm 17.12]{rudin} or \citep{butzer1971hilbert}). We will compare the radial limits of $f$ and $g$ and show that they coincide a.e., but first we need another lemma.

\begin{lemma} \label{}
$Z_{\tau_U}$ and $Z_{\tau_V}$ agree in distribution with $f(X)$ and $g(X)$ respectively, where $X$ is a r.v. uniformly distributed on $\{|z|=1\}$.
\end{lemma}

{\bf Proof:} Note that in applying $f$ and $g$ to $X$, we are making use of the radial limits defined above. We will prove the statement for $f$. Let $r_n$ be any sequence in $(0,1)$ which increases to 1 as $n \to \ff$, and let $\tau_n = \inf \{t > 0: |Z_t|=r_n\}$. By standard martingale theory (see for instance \cite{willy}), since $f(Z_{\tau_n})$ is a martingale bounded in $L^p$ we are guaranteed the existence of a limiting r.v. $M_\ff$ such that $E[|f(Z_{\tau_n}) - M_\ff|^p] \to 0$. Therefore $f(Z_{\tau_n})$ converges to $M_\ff$ in distribution. On the other hand, $f(Z_{\tau_n})$ is equal in distribution to $f(X_n)$, where $X_n$ is any r.v. uniformly distributed on $\{|z|=r_n\}$. Let us choose $X_n$ and $X$ as follows. Let the probability space in question be the interval $[0,2\pi)$, with probability measure given by Lebesgue measure divided by $2 \pi$. For $\omega$ in the probability space, let $X_n(\omega) = r_n e^{i \omega}$, and similarly $X_n(\omega) = e^{i \omega}$. By \cite[Thm 17.12]{rudin}, we have

\begin{equation} \label{}
\lim_{n \to \ff} \int_{0}^{2\pi} |f(r_ne^{i \th}) - f(e^{i \th})| d \th \to 0.
\end{equation}

Thus, $E[|f(X_n)-f(X)|] \to 0$, which implies that $f(X_n)$ converges to $f(X)$ in distribution. However, $f(X_n)$ and $f(Z_{\tau_n})$ have the same distribution, and therefore $M_\ff$ and $f(X)$ agree in distribution. Now, $f(Z_t)$ is a time-changed Brownian motion, and therefore $f(Z_{\tau_n}) = \hat Z_{\sigma(\tau_n)}$, where $\sigma$ denotes the time-change and $\hat Z$ is a Brownian motion. By monotone convergence, $\sigma(\tau_n) \nearrow \tau_U$, and thus $f(Z_{\tau_n})$ converges a.s. to $\hat Z_{\tau_U}$. It follows that $\hat Z_{\tau_U}$ is equal in distribution to $f(X)$. \qed


\vski

$\De$-convexity and symmetry imply that $Re(f(e^{i\th}))$ and $Re(g(e^{i\th}))$ are a.e. even functions on $[-\pi, \pi]$ and non-increasing on $[0, \pi]$. Since $P(U \in \{e^{i\th}: -\th_0 < \th < \th_0\}) = \frac{\th_0}{\pi}$ for $\th_0 \in (0,\pi]$, it follows that for a.e. $\th$ we must have $Re(f(e^{i\th})) = r$, where $r$ is such that $\mu[r, +\ff) = \frac{\th}{\pi}$, and the same must hold for $Re(g(e^{i\th}))$. We see that $Re(f)$ and $Re(g)$ agree a.e. on $\{|z|=1\}$, and since $Im(f), Im(g)$ are obtained from these by the periodic Hilbert transform (see Section \ref{hilbert}) we see that $f$ and $g$ agree a.e. on $\{|z|=1\}$. $f(z)$ and $g(z)$ for $z \in \DD$ can be obtained from their boundary values via the Poisson integral formula (\cite[Cor. 17.12]{rudin}), and thus $f$ and $g$ agree. Theorem \ref{thm:The-distribution-} is proved. \qed 

\vski

None of the three conditions in the theorem can be omitted. For example, suppose that $U = \CC \backslash \{|Re(z)|\leq 1, |Im(z)|\geq 1\}$. $U$ is symmetric and $\De$-convex, but $E[(\tau_{\Omega})^{p/2}]=\ff$ for $p \geq 1$. Since $Re(Z_{\tau_{\Omega}})$ is a measure of bounded support, it will generate by Gross' method a domain $\Omega$ such that $E[(\tau_{\Omega})^{p/2}]<\ff$ for all $p$, and this can therefore not be equal to $U$. An example which is symmetric and has finite $p$-th moment for all $p$ but which lacks $\De$-convexity is displayed in Figure 2 of \citep{gross2019conformal}, and it is pointed out there that uniqueness fails. It is similarly easy to construct a domain which is $\De$-convex and has finite $p$-th moment for all $p$ but which is not symmetric, and again uniqueness fails.

\vski

Another example that may be worth noting can be found in the next section, as it is shown there that the parabola and infinite strip lead to the same distribution $\mu$. This does not contradict our result, since the parabola is neither $\De$-convex nor symmetric, but it is interesting to note that both of these domains are convex, and that therefore convexity does not seem to be the correct condition for uniqueness.

\section{Examples} \label{examples}

In this section, we consider a series of domains and the corresponding distributions of $Re(Z_\tau)$. In all cases that we consider the boundary of the domain will be well behaved and we will be able to find a p.d.f. of the distribution of $Z_{\tau}$ on the boundary. By this we mean that we can find a function, $\rho_{a}^{Z_{\tau}}(z)$, defined for $z$ on $\partial U$ such that for any interval $I$ on the boundary of $U$, we have $P_a(Z_t \in I) = \int_{b}^{c}\rho_{a}^{Z_{\tau}}(z(s))ds$, where $z(s)$ is a parameterization of $\partial U$ with $|z'(s)| = 1$ and $z((b,c))=I$. We will use analytic functions and the conformal invariance of Brownian motion as our primary tool; finding exit distributions in this manner has previously been considered in \cite{markowdist}, and following the convention there we will use the notation $\rho_{a}^{Z_{\tau}}(z) ds$ to denote this density, with the $ds$ to indicate that the curve $z(s)$ is parameterized by arclength.

If we have found the p.d.f of $Z_{\tau}$ on $\partial U$,
then we can deduce the p.d.f's of $X_{\tau}$ and $Y_{\tau}$ provided
that the boundary of the domain is smooth enough in the sense that,
locally around $z=x+yi$, we have
\begin{equation}
y=\varphi_{z}(x)\label{y=00003D=00005Cphi(x)}
\end{equation}
 for some differentiable bijective function $\varphi_{z}$. To see how, let $x$ be an element of $\{\mathcal{R}e(z)|\,z\in\partial D\}$.
Since a positive infinitesimal element $dz\in\partial D$ is expressed
as $dz=\sqrt{dx^{2}+dy^{2}}$, then
\begin{equation}
\begin{alignedat}{1}\rho_{\mathcal{R}e(a)}^{X_{\tau}}(x)dx & =\sum_{\mathcal{R}e(z)=x}\rho_{a}^{Z_{\tau}}(z)dz\\
 & =\sum_{\mathcal{R}e(z)=x}\rho_{a}^{Z_{\tau}}(x+yi)\sqrt{dx^{2}+dy^{2}}\\
 & =\sum_{\mathcal{R}e(z)=x}\rho_{a}^{Z_{\tau}}(x+\varphi_{z}(x)i)\sqrt{1+\varphi_{z}'(x)^{2}}dx
\end{alignedat}
\label{X}
\end{equation}
Finally we get
\begin{equation}
\begin{alignedat}{1}\rho_{\mathcal{R}e(a)}^{X_{\tau}}(x) & =\sum_{\mathcal{R}e(z)=x}\rho_{a}^{Z_{\tau}}(x+\varphi_{z}(x)i)\sqrt{1+\left\{ {\textstyle \frac{d\varphi_{z}}{dx}(x)}\right\} ^{2}}\\
\rho_{\mathcal{I}m(a)}^{Y_{\tau}}(y) & =\sum_{\mathcal{I}m(z)=y}\rho_{a}^{Z_{\tau}}(\varphi_{z}^{-1}(y)+yi)\sqrt{1+\left\{ {\textstyle \frac{d\varphi_{z}^{-1}}{dy}(y)}\right\} ^{2}}
\end{alignedat}
\label{real =000026 imaginary}
\end{equation}
Notice that that both sets $\{z|\mathcal{R}e(z)=x\}$ and $\{z|\mathcal{I}m(z)=y\}$
are countable due to (\ref{y=00003D=00005Cphi(x)}), which justifies
the sum symbols in (\ref{X}). This proves the formula for the distribution of $X_{\tau}$, and $Y_{\tau}$ can of course be obtained similarly.
The following diagram, which should be viewed at the infinitesimal level, provides the intuitive justification for the formulas.

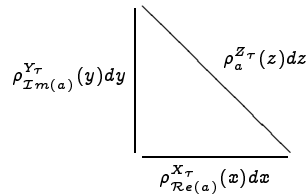
\[
\xymatrix{\ar@{-}[rrdd]^{\rho_{a}^{Z_{\tau}}(z)dz}\\
\\
\ar@{-}[uu]^{\rho_{\mathcal{I}m(a)}^{Y_{\tau}}(y)dy} &  & \ar@{-}[ll]^{\rho_{\mathcal{R}e(a)}^{X_{\tau}}(x)dx}
}
\]

Before going through examples, we recall
the following lemma which will be used across the rest of the work.

\begin{lemma}
If $X$ is a random variable with c.d.f $F_{X}$ and $f$
is a function where each point in its range has at most a countable
number of pre-images, then
\[
dF_{f(X)}(x)=\sum_{t\in f^{-1}\{x\}}d[F_{X}(t)]
\]
Furthermore, if $X$ has a p.d.f., say $\rho_{X}$, and $f$ is differentiable,
then
\[
dF_{f(X)}(x)=\sum_{t\in f^{-1}\{x\}}\frac{\rho_{X}(t)}{|f'(t)|}dx
\]
\end{lemma}

A proof in the case that $f$ is analytic (which is what we use in this paper) can be found in \cite{markowdist}.

\subsection{Unit disc.}

If $Z_{t}$ starts at zero at stopped at $\tau_{\mathbb{D}}$ then
due to the rotational invariance of the Brownian motion $Z_{\tau_{\mathbb{D}}}$is uniformly distributed
on the circle, i.e
\[
\rho_{0}^{Z_{\tau_{\mathbb{D}}}}(e^{\theta i})=\frac{1}{2\pi}
\]
Using the unit circle equation $x^{2}+y^{2}=1$, we extract the distributions
of $X_{\tau_{\mathbb{D}}}$ and $Y_{\tau_{\mathbb{D}}}$ on $(-1,1)$:
\[
\begin{alignedat}{1}\rho_{0}^{X_{\tau_{\mathbb{D}}}}(x) & \overset{\eqref{X}}{=}\sum_{z\in\{x\pm i \sqrt{1-x^{2}}\}}\rho_{a}^{Z_{\tau}}(z)\\
 & =\frac{1}{2\pi}\sqrt{1+({\textstyle \frac{x}{1-x^{2}}})^{2}}+\frac{1}{2\pi}\sqrt{1+({\textstyle -\frac{x}{1-x^{2}}})^{2}}\\
 & =\frac{1}{\pi\sqrt{1-x^{2}}}
\end{alignedat}
\]
Similarly for $\rho_{0}^{Y_{\tau_{\mathbb{D}}}}(y)$. We remark that
$X_{\tau_{\mathbb{D}}}$and $Y_{\tau_{\mathbb{D}}}$follow the Arc-sine
law on $(-1,1)$. If the starting point is $a=u+vi\neq0$, then the
distribution of $Z_{\tau_{\mathbb{D}}}$ is given by
\[
\rho_{a}^{Z_{\tau_{\mathbb{D}}}}(e^{\theta i})d\theta=\frac{1-|a|^{2}}{2\pi|1-\overline{a}e^{\theta i}|^{2}}d\theta
\]
(See \citep{markowdist}). Using the coordinates expressions
$(x,y)=(\cos\theta,\sin\theta)$, we find the distributions of $X_{\tau_{\mathbb{D}}}$
and $Y_{\tau_{\mathbb{D}}}$:
\[
\rho_{u}^{X_{\tau_{\mathbb{D}}}}(x)={\textstyle \frac{1-|a|^{2}}{2\pi\sqrt{1-x^{2}}}}\left({\textstyle \frac{1}{|1-\overline{a}(x+\sqrt{1-x^{2}}i|^{2}}}+{\textstyle \frac{1}{|1-\overline{a}(x-\sqrt{1-x^{2}}i|^{2}}}\right)
\]
and
\[
\rho_{v}^{Y_{\tau_{\mathbb{D}}}}(y)={\textstyle \frac{1-|a|^{2}}{2\pi\sqrt{1-y^{2}}}}\left({\textstyle \frac{1}{|1-\overline{a}(\sqrt{1-y^{2}}+yi|^{2}}}+{\textstyle \frac{1}{|1-\overline{a}(-\sqrt{1-y^{2}}+yi|^{2}}}\right).
\]
In particular we recover $\rho_{\mathcal{I}m(a)}^{Y_{\tau_{\mathbb{D}}}}(y)=\rho_{\mathcal{R}e(-ai)}^{X_{\tau_{\mathbb{D}}}}$.

\subsection{Parabola}

Let $S$ be the horizontal strip $\{z,\,1<\text{Im}(z)<-1\}$ and
$\mathcal{P}=f(S)$ where $f:z\longmapsto z^{2}$.The map $f$ is
not conformal as it is $2$ to $1$, however it maps the $\partial S$
to $\partial\mathcal{P}$. That is
\[
\partial\mathcal{P}=\{(u,v)|u=x^{2}-1,\,v=\pm2x,x\in\mathbb{R}\}
\]
so $\mathcal{P}$ is the area limited by the parabola of the equation
\begin{equation}
x=\frac{y^{2}}{4}-1.\label{u,v}
\end{equation}
The p.d.f of $Z_{\tau_{S}}$ starting from the origin is given by
\begin{equation}
\rho_{0}^{\tau_{S}}(z=x+i)=\frac{\mathrm{sech}(\frac{\pi}{2}x)}{2}\label{pdf}
\end{equation}
where $\mathrm{sech}(z)=\frac{2}{e^{z}+e^{-z}}$ is the hyperbolic
secant function. The density $\rho_{0}^{\tau_{S}}$ is equally shared
between the two horizontal lines of the boundary of the strip because
of symmetry. (\ref{pdf}) can be proved by conformal invariance (\cite{markowdist}) or as a consequence of the optional stopping theorem (\cite{chinjungmark}).

The expression of $\rho_{0}^{Z_{\tau_{\mathcal{P}}}}$ is as follows
\[
\begin{alignedat}{1}\rho_{0}^{Z_{\tau_{\mathcal{P}}}}(w=u+vi) & =\sum_{z\in f^{-1}\{w\}}{\textstyle {\displaystyle \frac{\rho_{0}^{Z_{\tau_{S}}}(z)}{\begin{vmatrix}f'(z)\end{vmatrix}}}}\\
 & ={\displaystyle \frac{\rho_{0}^{Z_{\tau_{S}}}(\frac{v}{2}+i)}{\begin{vmatrix}f'(\frac{v}{2}+i)\end{vmatrix}}}+{\displaystyle \frac{\rho_{0}^{Z_{\tau_{S}}}(-\frac{v}{2}+i)}{\begin{vmatrix}f'(-\frac{v}{2}+i)\end{vmatrix}}}\\
 & =2{\displaystyle \frac{\rho_{0}^{Z_{\tau_{S}}}(\frac{v}{2}+i)}{\begin{vmatrix}f'(\frac{v}{2}+i)\end{vmatrix}}}\\
 & ={\displaystyle \frac{\mathrm{sech}(\frac{\pi}{4}v)}{4\sqrt{\frac{v^{2}}{4}+1}}}
\end{alignedat}
\]
where $w\in\partial\mathcal{P}$. Via (\ref{real =000026 imaginary}),
we get for $(u,v)\in(-1,+\infty)\times\mathbb{R}$
\[
\begin{alignedat}{1}\rho_{0}^{X_{\tau_{\mathcal{P}}}}(u) & =\rho_{0}^{Z_{\tau_{\mathcal{P}}}}(u+\sqrt{4u+4}i){\textstyle \sqrt{1+\frac{4}{4u+4}}}+\rho_{0}^{Z_{\tau_{\mathcal{P}}}}(u-\sqrt{4u+4}i){\textstyle \sqrt{1+\frac{4}{4u+4}}}\\
 & ={\displaystyle \frac{\mathrm{sech}(\frac{\pi}{2}\sqrt{u+1})}{2\sqrt{u+1}}}
\end{alignedat}
\]
and

\[
\begin{alignedat}{1}\rho_{0}^{Y_{\tau_{\mathcal{P}}}}(v) & =\rho_{0}^{Z_{\tau_{\mathcal{P}}}}({\textstyle \frac{v^{2}}{4}-1+vi}){\textstyle \sqrt{1+\frac{v^{2}}{4}}}\\
 & ={\displaystyle \frac{\mathrm{sech}(\frac{\pi}{4}v)}{4}}
\end{alignedat}
\]

It is a surprising fact that this agrees with the density obtained from a strip; however, as mentioned in the previous section this does not contradict Theorem \ref{thm:The-distribution-} since it is the distribution of $Im(Z_{\tau_{\mathcal{P}}})$, and $\mathcal{P}$ is not symmetric or $\De$-convex with respect to the imaginary axis. 

\subsection{Ellipse of the form $\frac{x^{2}}{\cosh^{2}R}+\frac{y^{2}}{\sinh^{2}R}=1$.}

Let $E$ be the centered ellipse of equation $\frac{x^{2}}{\cosh^{2}R}+\frac{y^{2}}{\sinh^{2}R}=1$
and run a Brownian motion $Z_{t}$ starting at zero, killed at $\tau_{E}$.
In order to find the c.d.f of $Z_{\tau_{E}}$, we give a holomorphic
map $f$ acting on the horizontal strip $S_{R}:=\{z,\,R<\text{Im}(z)<-R\}$
where $R$ is a positive constant to be determined later. It turns
out that $f(z)=\sin(z)$ is a good map for this purpose, and this
is how it works: for $z=x+Ri$ we have
\[
\begin{alignedat}{1}\sin(z) & =\frac{e^{(x+Ri)i}-e^{-(x+Ri)i}}{2i}\\
 & =\frac{e^{-R}(\cos x+\sin xi)-e^{R}(\cos x-\sin xi)}{2i}\\
 & =\left({\textstyle \frac{e^{R}+e^{-R}}{2}}\right)\sin x+\left({\textstyle \frac{e^{R}-e^{-R}}{2}}\right)\cos xi\\
 & =\cosh R\sin x+\sinh R\cos xi
\end{alignedat}
.
\]
So if we set $\sin z=u+vi$ then $\frac{u^{2}}{\cosh^{2}R}+\frac{v^{2}}{\sinh^{2}R}=1$
and therefore $\sin z\in E_{a,b}$ where $a=\cosh R$ and $b=\sinh R$.
Now let $w\in\partial E_{a,b}$ and $\rho_{E}(w)$ be the p.d.f of
$Z_{\tau_{E}}$, then
\[
\begin{alignedat}{1}\rho_{E}(w=u+vi) & dw=\sum_{z\in f^{-1}\{w\}}{\textstyle {\textstyle \frac{\rho_{Z_{\tau}}(z)}{\begin{vmatrix}\cos(z)\end{vmatrix}}}}{\textstyle dw}\\
 & =\sum_{z\in f^{-1}\{w\}}{\textstyle {\textstyle {\textstyle \frac{\mathrm{sech}(\frac{\pi x}{2R})}{2R\begin{vmatrix}\cos(z)\end{vmatrix}}}}}dw\\
 & =\sum_{n\in\mathbb{Z}}{\textstyle \frac{\mathrm{sech}(\frac{\pi}{2R}\arcsin(\frac{u}{\cosh R})+\frac{n\pi^{2}}{R})}{2R\begin{vmatrix}\cos(\arcsin(\frac{u}{\cosh R})+2n\pi)\end{vmatrix}}}dw\\
 & ={\textstyle \frac{1}{2R\begin{vmatrix}\cos(\arcsin(\frac{u}{\cosh R}))\end{vmatrix}}}\sum_{n\in\mathbb{Z}}\mathrm{sech}({\textstyle \frac{\pi}{2R}}\arcsin({\textstyle \frac{u}{\cosh R}})+{\textstyle \frac{n\pi^{2}}{R}})dw\\
 & ={\textstyle \frac{\cosh R}{2R\sqrt{\cosh^{2}R-u^{2}}}}\sum_{n\in\mathbb{Z}}\mathrm{sech}({\textstyle \frac{\pi}{2R}}\arcsin({\textstyle \frac{u}{\cosh R}})+{\textstyle \frac{n\pi^{2}}{R}})dw
\end{alignedat}
\]

\subsection{Right part of the Hyperbola $x^{2}-y^{2}=1$.}

If $R:=\{z|\mathcal{R}e(z)>1\}$ then the p.d.f of $Z_{\tau_{R}}$
started at $a=\delta+\eta i\in R$ is given by \citep{markowdist}
\[
\rho_{a}^{\tau{}_{R}}(1+yi)=\frac{(\delta-1)}{\pi|1+iy-a|^{2}}dy.
\]
The square function $s:z\longmapsto z^{2}$ maps the right part limited
by the hyperbola $x^{2}-y^{2}=1$, say $H$, to $R$. Therefore for
every $z=x+yi\in\partial H$
\[
\begin{alignedat}{1}\rho_{\sqrt{a}}^{\tau_{H}}(z)dz & \overset{{\scriptscriptstyle z^{2}=1+vi}}{=}|s'(\sqrt{1+vi})|\rho_{a}^{\tau{}_{R}}(1+vi)\\
 & =2(\delta-1)\frac{\sqrt{x^{2}+y^{2}}}{\pi|1+vi-a|^{2}}dz\\
 & =2(\delta-1)\frac{\sqrt{x^{2}+y^{2}}}{\pi|1-a+2xyi|^{2}}dz
\end{alignedat}
.
\]
In particular if $a$ is real, and by using the relation $x^{2}-y^{2}=1$,
we get the densities of $X_{\tau_{H}}$ and $Y_{\tau_{H}}$:
\[
\begin{alignedat}{1}\rho_{\sqrt{a}}^{X_{\tau_{H}}}(x) & =\frac{2(\delta-1)}{\pi}\frac{2x^{2}-1}{\sqrt{x^{2}-1}}\left\{ \frac{1}{|2x\sqrt{x^{2}-1}i+1-a|^{2}}+\frac{1}{|2x\sqrt{x^{2}-1}i-1+a|^{2}}\right\} \\
\rho_{\sqrt{a}}^{Y_{\tau_{H}}}(y) & =\frac{2(\delta-1)}{\pi}\frac{2y^{2}+1}{\sqrt{1+y^{2}}|2y\sqrt{y^{2}+1}i+1-a|^{2}}
\end{alignedat}
\]

\section{Concluding remarks}

We do not know whether Theorem \ref{newguy} holds for $\frac{1}{2} \leq p \leq 1$. There are many difficulties to proving the result in this range. One is that the analogue of Theorem \ref{Hilb bound} does not hold, even for $p=1$; for a counterexample, see \cite[p. 212]{king2009hilbert}. Furthermore $H^p$ and $L^p$ are not as well behaved for $p<1$; their respective norms are not true norms, for instance, as the triangle inequality fails. In any event, regardless of the veracity of the theorem for $\frac{1}{2} \leq p \leq 1$, one should certainly exercise extreme caution in attempting to extend it to $p<\frac{1}{2}$. This is because for any simply connected domain $\Omega$ strictly smaller than $\CC$ itself we have $E[(\tau_{\Omega})^{p/2}] < \ff$ for any $p< \frac{1}{2}$; this is proved in \citep{burkholder1977exit}. Thus a measure with infinite $p$-th moment for some $p<\frac{1}{2}$ cannot correspond in this manner to a simply connected domain.

\vski

The question posed by Gross in \cite{gross2019conformal} on how properties of $\mu$ are reflected in the geometry of $\Omega$ is, in our opinion, an interesting one. Gross proposed finding a condition which forced $\Omega$ to be convex; this appears difficult, especially considering that according to Gross' simulations the domain corresponding to a Gaussian is not convex. We would like therefore to suggest several weaker properties that $\Omega$ might have, and propose that finding sufficient conditions on $\mu$ for these might be interesting problems.

\begin{itemize} \label{}

\item $\Omega$ is starlike with respect to 0.

\item $\sup_{z \in \Omega} |\mathcal{I}m(z)|<\ff$. That is, $\Omega$ is contained in an infinite horizontal strip. Note that this would imply that all moments of $\mu$ are finite, because all moments of the exit time of a strip are finite, but that this is not sufficient: if $\Omega$ is the parabolic region $\{x>y^2-1\}$, then all moments of $\tau_\Omega$ are finite (proof: $\Omega$ can fit inside a rotated and translated wedge $W_\alpha$ with arbitrarily small aperture $\alpha$, and therefore its exit time is dominated by that of the wedge, which can have finite $p$-th moment for as large $p$ as we like) but $\sup_{z \in \Omega} |\mathcal{I}m(z)|=\ff$.

\item $\limsup_{|\mathcal{R}e(z)| \to \ff, z \in \Omega} |\mathcal{I}m(z)| = 0$.

\end{itemize}

\section{Acknowledgements}

We would like to thank Zihua Guo, Paul Jung, and Wooyoung Chin for helpful comments.

\bibliographystyle{plain}
\phantomsection\addcontentsline{toc}{section}{\refname}\bibliography{Maheref}

\end{document}